\documentclass{amsart}
\usepackage{amssymb}
\usepackage{graphicx}
\usepackage[shortlabels]{enumitem}
\usepackage{multirow}
\usepackage{enumitem}
\usepackage[T1]{fontenc}
\usepackage{amsmath,color}
\usepackage{hyperref}
\usepackage{url}
\usepackage{longtable}
\newtheorem{thm}{Theorem}[section]
\newtheorem{prp}[thm]{Proposition}
\newtheorem{lema}[thm]{Lemma}

\newtheorem{cor}[thm]{Corollary}
\newtheorem{rem}[thm]{Remark}
\newtheorem{ques}[thm]{Question}
\newtheorem{example}[thm]{Example}
\newtheorem{con}[thm]{Conjecture}

{% This is the begin code
\newcommand{\G}{\Gamma}

\title{Graphs with two main and two plain eigenvalues}

\author[S. Hayat]{ Sakander Hayat$^\spadesuit$}
\thanks{$^\spadesuit$S.H is supported by a CAS-TWAS president's fellowship at USTC, China}
\author[M. Javaid]{ Muhammad Javaid$^\clubsuit$ }
\thanks{$^\clubsuit$M.J is supported by Chinese Academy of Science, President's International Fellowship Initiative (CAS-PIFI), Beijing, China. Grant No: 2015 PM 035}
\author[J. H. Koolen]{ Jack H. Koolen$^{\diamondsuit,\star}$}
\thanks{$^\diamondsuit$J.H.K. is partially supported by the National Natural Science Foundation of China (Nos. 11471009 and 11671376)}
\thanks{$^\star$Corresponding author}

\address{School of Mathematical Sciences,
University of Science and Technology of China,
Hefei, Anhui, 230026, P.R. China}
\email{sakander@mail.ustc.edu.cn}

\address{School of Mathematical Sciences,
University of Science and Technology of China,
Hefei, Anhui, 230026, P.R. China}
\email{javaidmath@gmail.com}

\address{Wen-Tsun Wu Key Laboratory of CAS, School of Mathematical Sciences,
University of Science and Technology of China,
Hefei, Anhui, 230026, P.R. China}
\email{koolen@ustc.edu.cn}

\begin{document}



\maketitle
\begin{center}
   \emph{In honor of Drago\v{s} Cvetkovi\'{c} on the occasion of his 75th birthday.}
\end{center}
%%%%%%%%%%%%%%%%%%%%%%%%%%%%%%%%%%%%%%%%%%%%%%%%%%%%%%%%%%%%%%%%%%%%%%%%%

\begin{abstract}
In this paper, we introduce the concepts of the plain eigenvalue, the main-plain index and the refined spectrum of graphs.
We focus on the graphs with two main and two plain eigenvalues and give some characterizations of them.
\end{abstract}
%%%%%%%%%%%%%%%%%%%%%%%%%%%%%%%%%%%%%%%%%%%%%%%%%%%%%%%%%%%%%%%%%%%%%%%%
\section{Introduction}
%%%%%%%%%%%%%%%%%%%%%%%%%%%%%%%%%%%%%%%%%%%%%%%%%%%%%%%%%%%%%%%%%%%%%%%%
For undefined notations, see next section.\\

In 1970, Doob \cite{D70} asked to study graphs with a few eigenvalues and he proposed at most five.
A connected regular graph with at most three distinct eigenvalues is known
to be strongly regular, and see, for example \cite{BV84} for a survey on strongly regular graphs.
The connected non-regular graphs with three distinct eigenvalues have been studied by, for example, De Caen, Van Dam \& Spence \cite{DVS99},
Bridges \& Mena \cite{BM81}, Muzychuk \& Klin \cite{MK98}, Van Dam \cite{V98} and Cheng, Gavrilyuk, Greaves \& Koolen \cite{CGGK15}.
The connected regular graphs with four distinct eigenvalues were studied by Van Dam \cite{V95}, Van Dam \& Spence \cite{VS98}
and Huang \& Huang \cite{HH17}, among others.
Cioab\u{a}, Haemers, Vermette \& Wong \cite{CHVW15} (resp. Cioab\u{a}, Haemers \& Vermette \cite{CHV16})
studied connected graphs with at most two eigenvalues not equal to $1$ and $-1$ (resp. $0$ and $-2$).
In this note, we look at a certain class of non-regular graphs with at most four distinct eigenvalues.\\

To be more precise, we first define the main and plain eigenvalues of a real symmetric matrix $T$.
An eigenvalue $\sigma$ of $T$ is said to be a \emph{main eigenvalue} (resp. \emph{plain eigenvalue})
if there exists an eigenvector $\mathbf{x}$ corresponding to $\sigma$,
such that $\mathbf{x}^{\top}\mathbf{j}\neq0$ (resp. $\mathbf{x}^{\top}\mathbf{j}=0$).
%Moreover, $\sigma$ is called \emph{plain} if there exists an
%eigenvector $\mathbf{x}$ corresponding to $\sigma$, such that $\mathbf{x}^{\top}\mathbf{j}=0$.
In this note, we will look at the class of graphs with two main and two plain eigenvalues
and show that there are many examples.
Note that the number of distinct valencies, for a connected graph with two main and two plain eigenvalues, is unbounded.\\

The following result is an immediate consequence of the Perron-Frobenius theorem \cite[Theorem 2.2.1]{BH12}.
\begin{lema}\cite{R07}\label{lem-sp-rad}
The spectral radius of a nonnegative irreducible real symmetric matrix is always simple and main.
\end{lema}

In 1999, De Caen \cite{DVS99} posed the following question:
Does a connected graph with three distinct eigenvalues has at most three
distinct valencies?
In a similar fashion to de Caen's question, we formulate the following conjecture.
\begin{con}\label{conj}
Let $\Gamma$ be a connected $t$-valenced graph with two main and two plain eigenvalues.
There exists a positive integer $C$ such that if $t\geq C$, then $\Gamma$ is a strong graph.
\end{con}
\noindent We wonder whether the Conjecture \ref{conj} is true for $C=4$.

%%%%%%%%%%%%%%%%%%%%%%%%%%%%%%%%%%%%%%%%%%%%%%%%%%%%%%%%%%%%%%%%%%%%%%%%%
\section{Preliminaries}\label{sec2}
%%%%%%%%%%%%%%%%%%%%%%%%%%%%%%%%%%%%%%%%%%%%%%%%%%%%%%%%%%%%%%%%%%%%%%%%%
Let $\G$ be a simple graph with vertex set $V(\G)$ and adjacency matrix $A$.
The \emph{complement} of $\Gamma$, is the graph $\overline{\Gamma}$, with the same vertex set
such that two distinct vertices are adjacent in $\overline{\Gamma}$ if and only if they
are not adjacent in $\Gamma$.
By the \emph{eigenvalues}, \emph{eigenvectors} and \emph{spectrum}
of $\Gamma$, we mean the eigenvalues, eigenvectors and spectrum of its adjacency matrix $A$.
We say that a graph $\Omega$ is \emph{cospectral} with $\G$ if their
spectra are same.
We will denote the matrix of all-ones, the identity matrix and the vector of all-ones
by $J$, $I$ and $\textbf{j}$ respectively.\\

We call a regular graph \emph{strongly regular} if there are constants
$a$ and $c$ such that every pair of distinct vertices has $a$ or $c$ common neighbors if
they are adjacent or non-adjacent respectively.
A strongly regular graph $\G$ with at least two vertices is called \emph{primitive}
if both $\G$ and its complement $\overline{\Gamma}$ are connected, otherwise \emph{imprimitive}.
%Moreover, $\G$ is called \emph{imprimitive}, if it is not primitive.
Note that the only imprimitive strongly regular graphs are the regular complete multipartite graphs.
We refer the Godsil \& Royle \cite[Chapter 10]{gr01} for a background on strongly regular graphs.\\

We assume that the graph $\G$ has $s$ distinct valencies $k_{1},\ldots,k_{s}$.
We write $V_i:=\{v\in V (\Gamma)\mid d_v=k_i\}$ and
$n_i:=\mid V_i\mid$ for $i\in\{1,\ldots,s\}$. Clearly the subsets $V_i$ partition the vertex set of
$\Gamma$ and this partition is called the \emph{valency partition} of $\Gamma$.
A graph $\Gamma$ is called \emph{$t$-valenced} if it has $t$ number of distinct valencies.
A $2$-valenced graph is also called a \emph{biregular} graph.
Let $\pi=\{\pi_1,\ldots,\pi_s\}$ be a partition of the vertex set of $\G$. For each vertex $x$ in $\pi_i$,
write $d_{x}^{(j)}$ for the number of neighbors of $x$ in $\pi_j$. Then we write $b_{ij}=\frac{1}{|\pi_i|}\sum_{x\in\pi_i}d_{x}^{(j)}$
for the average number of neighbors in $\pi_j$ of the vertices in $\pi_i$. The matrix $B_{\pi}:=(b_{ij})$ is called the \emph{quotient matrix}
of $\pi$ and $\pi$ is called \emph{equitable} if for all $i$ and $j$, we have $d_{x}^{(j)}=b_{ij}$ for each $x\in\pi_i$.
A graph $\G$ is called an \emph{equitable} graph if its valency partition
is equitable.
%An \emph{equitable biregular} graph
%is a $2$-valenced graph whose valency partition is equitable.
Following Muzychuk and Klin \cite{MK98}, we say that a  biregular graph is \emph{strongly biregular}
if it has exactly three distinct eigenvalues.
%The \emph{cone over a graph} $\Gamma$ is the graph with vertex set $\{ \infty \} \cup V(\G)$
%such that $\infty$ is adjacent to all vertices of $\Gamma$. When $\Gamma$ is not specified, we call it a \emph{cone}.
The \emph{multicone over a graph} $\G$ is the graph with vertex set $S\cup V(\G)$, where $S$ is the set of isolated vertices
with $|S|\geq1$, such that every vertex of $S$ is adjacent to all vertices of $\Gamma$.
We just call it a \emph{multicone} when $\G$ is not specified. When $|S|=1$, we call it a \emph{cone}.

%%%%%%%%%%%%%%%%%%%%%%%%%%%%%%%%%%%
\section{The main-plain index and the refined spectrum}\label{sec3}
%%%%%%%%%%%%%%%%%%%%%%%%%%%%%%%%%%%
In this section, we will introduce some new terminologies which are subsequently used in the next sections.\\

For an eigenvalue $\theta$ of a real symmetric matrix $T$ with multiplicity $m$,
the \emph{plain multiplicity} $p$ of $\theta$ is defined by
$\dim(\xi_{T}(\theta)\cap\mathbf{j}^{\top})$. The set $\{\mathbf{x}\mid \mathbf{x}\in(\xi_{T}(\theta)\cap\mathbf{j}^{\top})\}$
is called the \emph{plain eigenspace} of $\theta$ and $\mathbf{x}$ is said to be the \emph{plain eigenvector} of $T$ corresponding to $\theta$.
Note that $p\leq m\leq p+1$ and $\theta$ is main if and only if
$m\neq p$ and plain if and only if $p\neq0$.
We can define \emph{refined spectrum} of $T$ by $\big(r,s;\mu_1,\ldots,\mu_{r};[\pi_{1}]^{p_{1}},\ldots,[\pi_{s}]^{p_{s}}\big)$,
where $\mu_i$ for $i=1,\ldots,r$ are the main eigenvalues of $T$ and $\pi_i$ for $i=1,\ldots,s$ are the plain eigenvalues
of $T$ with corresponding plain multiplicities $p_i$.
Note that by the above discussion, it is easy to obtain the spectrum of $T$ from its refined spectrum.
The pair $(r,s)$ is said to be the \emph{main-plain index} of $T$.
We define by $\xi_{T}(\theta)$ the eigenspace of $\theta$ for $T$. When $T$ is assumed to be the adjacency matrix of a graph $\G$, we
usually just write $\xi(\theta)$, if no confusion can occur.\\

For the adjacency matrix of a graph $\G$, if $\G$ has refined spectrum $\big(r,s;\mu_1,\ldots,\mu_{r};$\\ $[\pi_{1}]^{p_{1}},\ldots,[\pi_{s}]^{p_{s}}\big)$,
then the complement $\overline{\G}$ of $\G$ has refined spectrum
$\big(r,s;\mu'_1,\ldots,\mu'_{r};$ \\ $[-\pi_{1}-1]^{p_{1}},\ldots,[-\pi_{s}-1]^{p_{s}}\big)$ for
some real numbers $\mu'_1,\ldots,\mu'_{r}$. In particular, if $\G$ has main-plain index $(r,s)$, then
$\overline{\G}$ has also main-plain index $(r,s)$.\\

Now we give an example of refined spectrum.
\begin{example}\label{example}
Let $\G=K_{l_1}\cup K_{l_2}\cup\ldots\cup K_{l_t}$ with $l_1\geq l_2\geq\ldots,\geq l_t\geq1$.
We write the multiset of the $l_{i}$'s by $\{[m_{1}]^{b_{1}},[m_{2}]^{b_{2}},\ldots,[m_{d}]^{b_{d}}\}$
where $\{l_1,l_2,\ldots,l_t\}=\{m_1,m_2,\ldots,m_d\}$, $b_{i}=|\{j\mid l_{j}=m_{i}\}|$ for $i=1,2,\ldots,d$
and $m_{1}>m_2>\ldots>m_d$.
Let $U=\{i\mid b_{i}\geq2\}=\{u_{1},u_{2},\ldots,u_{k}\}$. Then the refined spectrum of $\G$ is
$(d,k+1;m_{1}-1,\ldots,m_{d}-1;
[-1]^{(l_1+\ldots+l_t)-t},[m_{u_1}-1]^{b_{1}-1},\ldots,[m_{u_k}-1]^{b_{k}-1})$.
\end{example}

Without proof we give the following trivial bound on the main-plain index of a graph.
\begin{lema}\label{tmain_tplain}
Let $\Gamma$ be a connected graph with distinct eigenvalues $\theta_{1},\theta_2,\ldots,\theta_d$ and
respective multiplicities $m_{1},m_{2},\ldots,m_{d}$.
Let $(r,s)$ be the main-plain index of $\G$ and $l:=|\{i\mid m_{i}\geq 2\}|$. Then $r\leq d$, $s\geq l$, and $d\leq r+s\leq l+d$
all hold.
\end{lema}

\begin{rem}
In Section \ref{sec51}, we will discuss a class of graphs with four distinct eigenvalues of which two are simple and main and
the remaining two eigenvalues are plain. This shows that the lower bound $r+s=d$ can be achieved.
\end{rem}
In Section \ref{sec54}, we will see that the upper bound can also be achieved.\\

Note that two cospectral graphs can have different number of main eigenvalues and thus different refined spectra.
For example, the cospectral pair $K_{1,4}\cup K_{1}$ and $C_{4}\cup 2K_{1}$ have a different number of main eigenvalues
as $K_{1,4}\cup K_{1}$ has refined spectrum
$(3,1;2,0,-2;[0]^{3})$ and $C_{4}\cup 2K_{1}$ has refined spectrum $(2,2;2,0;[0]^{3},[-2]^{1})$.\\

In view of the above discussion, we would like to ask the following question.
\begin{ques}
Suppose that the graphs $\G_{1}$ and $\G_{2}$ have the same refined spectra.
Are their complements cospectral?
\end{ques}
\noindent In general, the answer is probably no.\\

Now we turn our attention to the graphs with two main eigenvalues.
The following result was shown by Hagos.
\begin{prp}\emph{(\cite[Theorem 2.1]{H02})}\label{hagos-main-ev}
The number of main eigenvalues of $\Gamma$ is the largest integer $k$ such that
the vectors $\mathbf{j},A\mathbf{j},A^{2}\mathbf{j},\ldots,A^{k-1}\mathbf{j}$
are linearly independent.
\end{prp}
As an immediate consequence we have the following well-known fact.
\begin{lema}\emph{\cite[Prop. 1.4]{R07}}\label{1-main}
A graph $\Gamma$ has one main eigenvalue if and only if $\Gamma$ is regular.
\end{lema}

Now we assume that, the graph $\G$ has two main eigenvalues.
If we denote $\mathbf{d}:=A\mathbf{j}$, then by Prop. \ref{hagos-main-ev},
the graph $\Gamma$ satisfies the following relation:
\begin{equation}\label{hagos-eqn}
A\mathbf{d}=a\mathbf{d}+b\mathbf{j},
\end{equation}
where $a,~b$ are some real numbers.
Note that $\Gamma$ can not be a regular graph.\\

Hayat et al. \cite[Theorem 2.1]{HKLQ16} showed that an equitable biregular graph
has precisely two main eigenvalues.
They also showed that the number of distinct valencies of a connected
graph with two main eigenvalues is unbounded so they are far from equitable and
biregular. Now we show that a biregular graph with two main eigenvalues
is equitable.
\begin{prp}\label{bireg-equitable}
Let $\Gamma$ be a biregular graph. Then $\Gamma$ has two main eigenvalues if and only $\Gamma$ is equitable.
\end{prp}
\begin{proof}
We only need to show the `if part' of the statement. Let $\Gamma$ be a biregular graph with
two main eigenvalues. Let $d_{1}$ and $d_{2}$ be the two distinct valencies of $\Gamma$.
Then, for some real numbers $a,~b$, the graph $\Gamma$ satisfies:
\begin{equation*}
A\mathbf{d}=a\mathbf{d}+b\mathbf{j}.
\end{equation*}
Now we have $\sum_{y:y\sim x} d_y = (A\mathbf{d})_x = ad_{x}+b$ for any vertex $x$ and $d_{y}\in\{d_{1},d_{2}\}$ for all $y$.
This means that $x$ has a constant number of neighbors with valency $d_1$ and the valency of $x$ only depends on
whether $d_x=d_1$ or $d_x=d_2$. This shows that the valency partition is equitable. This completes the proof.
\end{proof}

Note that the cone over a graph $\Omega$ has adjacency matrix
$A=\left(\begin{array}{cc}
0 & \mathbf{j}^{\top} \\
\mathbf{j} & A'\\
\end{array}
\right),$
where $A'$ is the adjacency matrix of $\Omega$. Now, for an eigenvalue $\theta$ of $\Omega$, if $\mathbf{x}$ is a plain
eigenvector of $\Omega$, then
$\left(\begin{array}{c}
0 \\
\mathbf{x}\\
\end{array}
\right)$
is a plain eigenvector of the cone over $\Omega$. Consequently we have the following result.
\begin{prp}
Let $\G$ be a graph with $r$ main eigenvalues. Then the cone over $\G$ has at most $r+1$ main eigenvalues.
\end{prp}
Note that if we start with $\Omega=K_{1}$, then the resulting cone has exactly one main eigenvalue. We wonder whether
such an example also exists for larger $r$.
%%%%%%%%%%%%%%%%%%%%%%%%%%%%%%%%%%%%%%%%%%%%
\section{Some characterizations}\label{sec4}
%%%%%%%%%%%%%%%%%%%%%%%%%%%%%%%%%%%%%%%%%%%%
In this section, we present some characterizations of graphs with a small number of main and plain eigenvalues.\\

Note that the connected graphs with one main and one plain eigenvalue are the complete graphs with at least two vertices,
as there are at most two distinct eigenvalues and thus the diameter is exactly one.
As a connected regular graph with three distinct eigenvalues is strongly regular (see for example
\cite[Lemma 10.2.1]{gr01}), we obtain:
\begin{lema}\label{1main_2plain}
A connected graph $\Gamma$ has one main eigenvalue and at most two plain eigenvalues if and only if $\Gamma$ is a
strongly regular graph.
\end{lema}

Note that the connected graphs with one main eigenvalue and three plain eigenvalues are exactly the connected regular
graphs with four distinct eigenvalues. These graphs are studied by Van Dam \cite{V95}, Van Dam \& Spence \cite{VS98}
and Huang \& Huang \cite{HH17}, among others.\\

The following lemma shows that the connected graphs with two main eigenvalues and one plain eigenvalue are the
nonregular complete bipartite graphs. We use \cite[Prop. 2]{V98} to show it.
\begin{lema}\label{2main_1plain}
A connected graph $\Gamma$ has two main eigenvalues and one plain eigenvalue if and only if $\Gamma$
is a nonregular complete bipartite graph.
\end{lema}
\begin{proof}
Note that $\Gamma$ has three distinct eigenvalues as $\Gamma$ is not complete.
Let $\theta_{0}>\theta_{1}>\theta_{2}$ be the
distinct eigenvalues of $\Gamma$. By Lemma \ref{lem-sp-rad}, $\theta_0$ is simple and main. Then the other
main eigenvalue must be simple, as otherwise we would have two plain eigenvalues.
By \cite[Prop. 2]{V98} and Lemma \ref{1-main},
$\Gamma$ is nonregular and complete bipartite. The converse is straightforward, see also, Example \ref{example}.
\end{proof}

The following result characterizes the disconnected graphs with two main
and two plain eigenvalues.
\begin{prp}\label{disconnected}
Let $\G$ be a disconnected graph with two main and two plain eigenvalues.
Then $\G$ is a member of one of the following families.
\begin{enumerate}
  \item[\emph{(i)}] The disjoint union of cliques $K_{l_1}\cup K_{l_2}\cup\ldots\cup K_{l_t}$,
  where $t\geq3$, $l_{1}\neq l_{2}$ and $l_2=l_3=\ldots=l_t$.
  \item[\emph{(ii)}] The disjoint union of isolated vertices and a regular complete multipartite graph.
  \item[\emph{(iii)}] The disjoint union of an isolated vertex and a strongly regular graph.
  \item[\emph{(iv)}] The disjoint union of two strongly regular graphs with different valencies and the same
  non-trivial eigenvalues.
\end{enumerate}
\end{prp}
\begin{proof}
The case when each connected component is a clique follows immediately from Example \ref{example}.
And this shows (i).\\

Next we may assume that there exists a connected component, say, $\G_1$ of $\G$ with at least three eigenvalues.
Let $\G_1$ have vertex set $V_1$ with $n_1$ vertices and let $\G_2$ be the induced subgraph on $V_2=V-V_1$ with
$n_{2}=n-n_{1}$ vertices.
Let $\rho_1$ and $\rho_2$ be the spectral radii of $\G_1$ and $\G_2$, respectively.
Note that $\rho_1>0$, as $\G_1$ is not complete.
Since $\G$ has at most four distinct eigenvalues and if $\G$ has exactly four distinct eigenvalues, then
the main eigenvalues are simple, which implies that $\rho_1$ and $\rho_2$ are the main eigenvalues. This implies that
$\G_1$ has exactly three distinct eigenvalues $\rho_1>\theta_1>\theta_2$ where $\theta_1\geq0$ and $\theta_2<-1$ both hold.
This means that $\G_1$ and $\G_2$ are both regular and that $\G_1$ is strongly regular.
If $\rho_2>0$ then $\G_2$ has at least one other eigenvalue $\theta_3\leq -1$ and as $|\{\rho_1,\rho_2,\theta_1,\theta_2,\theta_3\}|\leq4$
we see that either $\rho_1=\rho_2$ or $\theta_3=\theta_2$. If $\rho_1=\rho_2$, then $|\{\rho_1,\rho_2,\theta_1,\theta_2,\theta_3\}|=3$
as then $\rho_1$ is main and plain. But this means that $\theta_1$ is a simple eigenvalue of $\G$, which implies that $\G_2$
has exactly two eigenvalues of which one is less than $-1$, a contradiction.
So we may assume $\rho_2>0$, $\rho_{1}\neq\rho_{2}$ and $\theta_3=\theta_2$. Now $\rho_1$ and $\rho_2$ are the main eigenvalues and $\theta_1$
and $\theta_2$ the plain eigenvalues. As $\rho_2>0$ and $\theta_2<-1$, we see that $\G_2$ is not complete and hence has three distinct eigenvalues.
This implies that $\rho_2$ is simple and $\G_2$ has distinct eigenvalue $\rho_2>\theta_1>\theta_2$. So we are in case (iv) of the proposition.\\

The last case we need to consider is $\rho_2=0$. Then we have either $\theta_1>0$ and $\G_2$ is an isolated vertex or $\theta_1=0$ and $\G_1$ is
regular and complete multipartite. This shows the cases (ii) and (iii).

\end{proof}

As a corollary we have:
\begin{cor}
Let $\G$ be a connected graph with two main and two plain eigenvalues. Assume that the complement of $\G$ is disconnected. Then
$\G$ is a member of one of the following families.
\begin{enumerate}
  \item[\emph{(i)}] The complete $t$-partite graph $K_{l_1,l_2,\ldots,l_t}$ where $t\geq3$, $l_{1}\neq l_{2}$ and $l_2=l_3=\ldots=l_t$.
  \item[\emph{(ii)}] The multicone over a regular complete multipartite graph.
  \item[\emph{(iii)}] The cone over a strongly regular graph.
  \item[\emph{(iv)}] The join of two strongly regular graphs with different valencies and the same
  non-trivial eigenvalues.
\end{enumerate}

\end{cor}
%%%%%%%%%%%%%%%%%%%%%%%%%%%%%%%%%%%%%%%%%%%%%%%%%%%%%%%%%%%%%%%%%%%%%%%%%%%%%%
\section{Examples and discussion}\label{sec5}
%%%%%%%%%%%%%%%%%%%%%%%%%%%%%%%%%%%%%%%%%%%%%%%%%%%%%%%%%%%%%%%%%%%%%%%%%%%%%%%
In this section, we propose some families of graphs with two main
and two plain eigenvalues.

%%%%%%%%%%%%%%%%%%%%%%%%%%%%%%%%%%%%%%%%%%%%%%%%%%%%%%%%%%%%%
\subsection{Strong graphs}\label{sec51}
%%%%%%%%%%%%%%%%%%%%%%%%%%%%%%%%%%%%%%%%%%%%%%%%%%%%%%%%%%%%%
The \emph{Seidel matrix} $S$ of a graph, with adjacency matrix $A$, is defined by $S=J-I-2A$.
A \emph{strong graph} is a graph such that its Seidel matrix $S$ satisfies $S^{2}\in\langle S,I,J\rangle$,
where $\langle \ldots\rangle$ denotes the $\mathbb{R}$-span. Seidel \cite{Sei68} showed that:
\begin{prp}\cite{Sei68}
Let $\G$ be a graph with Seidel matrix $S$. Then $\G$ is strong if and only if at least one of the following
holds:
\begin{itemize}
\item[\emph{(i)}] $G$ is strongly regular.
\item[\emph{(ii)}] $S$ has exactly two distinct eigenvalues.
\end{itemize}
\end{prp}

Hayat et al. \cite{HKLQ16} showed the following result.
\begin{thm}\cite{HKLQ16}\label{thm1}
Let $\G$ be an $n$-vertex non-regular strong graph.
If $\G$ has Seidel eigenvalues $[-1-2\theta_{0}]^{m_{0}},[-1-2\theta_{1}]^{m_{1}}$,
then it has four distinct (adjacency) eigenvalues $\mu_{0},\mu_{1},[\theta_{0}]^{m_{0}-1},[\theta_{1}]^{m_{1}-1}$
of which $\mu_{0}$ and $\mu_{1}$ are main eigenvalues and $\theta_{0},\theta_{1}$ are plain eigenvalues. The main eigenvalues
$\mu_{0},\mu_{1}$ are uniquely determined by $\theta_{0},~\theta_{1},~m_{0},~m_{1},~n$ and the number of edges.
\end{thm}

Thus, the non-regular strong graphs are the graphs with two main and two plain eigenvalues.

%%%%%%%%%%%%%%%%%%%%%%%%%%%%%%%%%%%%%%%%%%%%%%%%%%%%%%%%%%%%%%%%%%%%%%%%%%%%%%%%
\subsection{Strongly biregular graphs}\label{sec54}
%%%%%%%%%%%%%%%%%%%%%%%%%%%%%%%%%%%%%%%%%%%%%%%%%%%%%%%%%%%%%%%%%%%%%%%%%%%%%%%%
Let $\G$ be a connected strongly biregular graph. By a result of Van Dam \cite{V98}, $\G$ is equitable and
thus has two main eigenvalues of which one is the spectral radius.
Consequently, $\G$ has two main and at most two plain eigenvalues.
Note that a strongly biregular graph $\G$ has exactly two main and two plain eigenvalues if and only if
$\G$ is non-bipartite.\\

Rowlinson \cite{R16} characterized the strongly biregular graphs among the graphs with three distinct eigenvalues. For the convenience
of the reader, we include a proof of it.
\begin{lema}\label{rowlinson}
Let $\G$ be a connected graph with three distinct eigenvalues. Then two of them are main if and only if $\G$ is strongly biregular.
\end{lema}
\begin{proof}
Let $\G$ have three distinct eigenvalues, say $\theta_0>\theta_1>\theta_2$. Then these exists a positive vector $\alpha$
such that
\begin{equation}\label{vandam-eqn}
(A-\theta_1I)(A-\theta_2I)=\mathbf{\alpha}\mathbf{\alpha}^{\top},\hspace{0.5cm}\textrm{and}\hspace{0.3cm}A\mathbf{\alpha}=\delta_{0}\mathbf{\alpha}
\end{equation}
both hold.
From this equation we obtain $d_{i}=\alpha_{i}^{2}-\theta_1\theta_2$, where $d_{i}$ is the valency of vertex $i$.
Let $\G$ have two main eigenvalues. Then by Eq. (\ref{hagos-eqn}) we obtain $A^{2}\mathbf{j}\in\langle\mathbf{d},\mathbf{j}\rangle$.
Now by Eq. (\ref{vandam-eqn}) $\mathbf{\alpha}(\mathbf{\alpha}^{\top}\mathbf{j})\in\langle\mathbf{d},\mathbf{j}\rangle$ and
$\mathbf{\alpha}^{\top}\mathbf{j}\neq0$. Accordingly there exist real numbers $a$ and $b$ such that
$\mathbf{\alpha}=a\mathbf{d}+b\mathbf{j}$. It follows that
\begin{equation*}
a\alpha_{i}^{2}-\alpha_{i}-a\theta_1\theta_2+b=0\hspace{0.5cm}(i=1,\ldots,n),
\end{equation*}
and hence that the $\alpha_i$'s take just two values. Consequently, by Eq. (\ref{vandam-eqn}), the graph $\G$ has exactly two valencies.
The other direction immediately follows from Prop. \ref{bireg-equitable}.
\end{proof}

Bridges \& Mena \cite{BM81} and De Caen et al. \cite{DVS99} give examples of connected graphs with three eigenvalues
and three distinct valencies. They have three main eigenvalues by Lemma \ref{rowlinson}. Thus the upper bound $r+s=l+d$ in Lemma
\ref{tmain_tplain} can be achieved.

%%%%%%%%%%%%%%%%%%%%%%%%%%%%%%%%%%%%%%%%%%%%%%%%%%%%%%%%%%%%%%%%%%%%%%%%%%%%%%%%
\subsection{Vertex-deleted subgraph of a strongly regular graph}\label{sec54}
%%%%%%%%%%%%%%%%%%%%%%%%%%%%%%%%%%%%%%%%%%%%%%%%%%%%%%%%%%%%%%%%%%%%%%%%%%%%%%%%

Let $\G$ be a connected strongly regular graph with parameters $(n,k,a,c)$ and distinct eigenvalues $k>\theta_{1}>\theta_{2}$.
Fix a vertex $u$ of $\G$.
Let $N_{u}$ be the set of neighbors of vertex $u$. Then the partition $\sigma=\{u,N_{u},V-(\{u\}\cup N_{u})\}$ is equitable with quotient matrix
$Q=\left(\begin{array}{ccc}
0 & k & 0\\
1 & a & k-a-1 \\
0 & c & k-c \\
\end{array}
\right).$
As $Q$ has constant row sum, the all-ones vector $\mathbf{j}$ is an eigenvector of $Q$. Let $\mathbf{x}_1$ and $\mathbf{x}_2$
be the other two eigenvectors of $Q$. Let $P$ be the characteristic matrix of $\sigma$. Then $\mathbf{w}_1=P\mathbf{x}_1$
and $\mathbf{w}_2=P\mathbf{x}_2$ are two eigenvectors of $\G$ with eigenvalues $\theta_1$ and $\theta_2$, respectively.
The graph $\G-u$ is equitable biregular and thus has two main eigenvalues.
Note that the plain eigenvectors $\mathbf{w}_{1}$ and $\mathbf{w}_{2}$ of $\G$ both give
eigenvectors of $\G-u$ by deleting $(\mathbf{w}_{1})_{u}$ and $(\mathbf{w}_{2})_{u}$.
Therefore, $\G-u$ also has two plain eigenvalues.

%%%%%%%%%%%%%%%%%%%%%%%%%%%%%%%%%%%%%%%%%%%%%%%%%%%%%%%%%%%%%%%%%%%%%%%%%%%%%%%%%%%
\subsection{An exceptional example}
%%%%%%%%%%%%%%%%%%%%%%%%%%%%%%%%%%%%%%%%%%%%%%%%%%%%%%%%%%%%%%%%%%%%%%%%%%%%%%%%%%%
In this short subsection, by switching, we mean the Seidel switching of graphs.\\

Let $\G$ be a strongly regular graph.
Let $U\subseteq V(\G)$ be such that the partition $\sigma=\{U,V-U\}$
is equitable. We switch the graph $\G$ with respect to $U$. Note that
the graph $\G^{U}$ is biregular and equitable. Thus, by Prop. \ref{bireg-equitable},
the graph $\G^{U}$ has two main eigenvalues.\\

We now give an example where $\G^{U}$ has two main and two plain eigenvalues.\\

\noindent\textbf{Example:}
If we switch $L(K_{6,6})$, which is a strongly regular graph with parameters $(36,10,4,2)$,
with respect to a Delsarte clique which is a $6$-clique. Then we obtain an equitable biregular graph with
spectrum $\{[7+\sqrt{129}]^{1}, [4]^{9}, [-2]^{25}, [7-\sqrt{129}]^{1}\}$ of which $4$ and $-2$
are the two plain eigenvalues and $7\pm\sqrt{129}$ are the two main eigenvalues.\\

\noindent We wonder whether there are more of these examples.

%\noindent\textbf{Example 2:}
%By switching the triangular graph $T(9)$ with respect to a Delsarte clique which is a $8$-clique,
%we obtain a strongly biregular graph with spectrum \\$\{[21]^{1}, [5]^{7}, [-2]^{28}\}$
%of which $5$ and $-2$ are the two plain eigenvalues whereas $21$ and $-2$ are the two main eigenvalues.


\begin{thebibliography}{9}

\bibitem{BM81}
W. G. Bridges and R. A. Mena, Multiplicative cones--a family of three eigenvalue graphs,
{\em Aequationes Math.}, 22 (1981), 208--214.

\bibitem{BH12}
A. E. Brouwer and W. H. Haemers, {\it Spectra of
Graphs}, Springer, Heidelberg, 2012.

\bibitem{BV84}
A. E. Brouwer and J. H. van Lint, Strongly regular graphs and partial geometries, In {\em Enumeration and Design: Papers from the conference on combinatorics held at the University of Waterloo, Waterloo, Ont., June 14-July 2, 1982} (Ed. D. M. Jackson and S. A. Vanstone). Toronto, Canada: Academic Press, pp. 85-122, 1984.

\bibitem{CHVW15}
S. M. Cioab\u{a}, W. H. Haemers, J. R. Vermette and W. Wong, The graphs with all but two eigenvalues equal to $\pm1$, {\em  J. Algebr. Comb.}, 41 (2015), 887--897.

\bibitem{CHV16}
S. M. Cioab\u{a}, W. H. Haemers and J. R. Vermette, The graphs with all but two eigenvalues equal to $-2$ or $0$, {\em  Des. Codes Cryptogr.}, in press.


\bibitem{CGGK15}
X.-M. Cheng, A. L. Gavrilyuk, G. R. W. Greaves, and J. H. Koolen, Biregular graphs with three eigenvalues, {\em  European J. Combin.}, 56 (2016), 57--80.

\bibitem{DVS99}
D. de Caen, E. R. van Dam, and E. Spence, A non-regular analogue of conference graphs, {\em J. Combin. Theory Ser. A}, 88 (1999), 194--204.

\bibitem{D70}
M. Doob, Graphs with a small number of distinct eigenvalues, {\em Ann. New York Acad. Sci.}, 175 (1970), 104--110.

\bibitem{gr01}
C. D. Godsil and G. Royle,
{\em Algebraic Graph Theory}, Springer-Verlag, Berlin, 2001.

\bibitem{H02}
E. M. Hagos, Some results on graph spectra, {\em  Linear Algebra Appl.}, 356 (2002), 103--111.

\bibitem{HKLQ16}
S. Hayat, J. H. Koolen, F. Liu and Z. Qiao, A note on graphs with exactly two main eigenvalues, {\em Linear Algebra Appl.}, 511 (2016), 318--327.

\bibitem{HH17}
X. Huang and Q. Huang, On regular graphs with four distinct eigenvalues, {\em Linear Algebra Appl.}, 512 (2017), 219--233.

%\bibitem{JN80}
%C. R. Johnson and M. Newman, A note on cospectral graphs, {\em J. Combin. Theory Ser. B}, 28 (1980), 96--103.


\bibitem{MK98}
M. Muzychuk and M. Klin, On graphs with three eigenvalues, {\em Discrete Math.}, 189 (1998), 191--207.

\bibitem{R16}
P. Rowlinson, On graphs with just three distinct eigenvalues, {\em Linear Algebra Appl.}, 507 (2016), 402--473.

\bibitem{R07}
P. Rowlinson, The main eigenvalues of a graph: A survey, {\em Appl. Anal. Discr. Math.}, 1 (2007), 445--471.

%\bibitem{Sei76}
%J. J. Seidel,
%A survey of two-graphs,
%{\em Colloquio Internazionale sulle Teorie Combinatorie} (Proceddings, Rome, 1973)
%Vol. I, pp. 481--511. Academia Nazionale dei Lincei, Rome, 1976.

\bibitem{Sei68}
J. J. Seidel, Strongly regular graphs with $(-1,1,0)$-adjacency matrix having eigenvalue 3, {\em  Linear Algebra Appl.,} 1 (1968), 281--298.


\bibitem{V98}
E. R. van Dam,
Nonregular graphs with three eigenvalues,
{\em J. Combin. Theory Ser. B}, 73 (1998), 101--118.

\bibitem{V95}
E. R. van Dam,
Regular graphs with four eigenvalues,
{\em Linear Algebra Appl.}, 226-228 (1995), 139--162.

\bibitem{VS98}
E. R. van Dam and E. Spence,
Small regular graphs with four eigenvalues,
{\em Discrete Math.}, 189 (1998), 233--257.


%\bibitem{VH3}
%E. R. van Dam and W. H. Haemers, Which graphs are determined by their spectrum?,
%{\em Linear Algebra Appl.}, 373 (2003), 241--272.


%\bibitem{VKX15}
%E. R. van Dam, J. H. Koolen and Z.-J. Xia,
%Graphs with many valencies and few eigenvalues,
%{\em Electron. J. Linear Algebra}, 28 (2015), 12--24.

\end{thebibliography}
\end{document}